\documentclass{amsart}

\usepackage{amsmath,amssymb,amsthm}

\theoremstyle{plain}
\newtheorem{theorem}{Theorem}
\newtheorem{lemma}{Lemma}
\newtheorem{corollary}{Corollary}

\theoremstyle{definition}
\newtheorem{definition}{Definition}
\newtheorem{example}{Example}

\usepackage{graphicx}
\usepackage{subcaption}
\usepackage{algorithm}
\usepackage{algorithmicx}
\usepackage[noend]{algpseudocode}

\MakeRobust{\Call}

\newcommand{\I}{\text{I}}
\newcommand{\II}{\text{II}}
\newcommand{\III}{\text{III}}
\newcommand{\IV}{\text{IV}}
\newcommand{\kmin}{k^\text{min}}
\newcommand{\Bcal}{\mathcal{B}}
\newcommand{\suchthat}{\; : \;}

\newcommand{\CommentLine}[1]{\State $\triangleright$ #1}

\title{Planar additive bases for rectangles}

\author{Jukka Kohonen}
\address{Department of Computer Science, University of Helsinki, P.O. Box 68, FI-00014 University of Helsinki, Finland}
\email{jukka.kohonen@helsinki.fi}

\author{Visa Koivunen}
\address{Department of Signal Processing and Acoustics, Aalto University, P.O. Box 15400, FI-00076 Aalto, Finland}
\email{visa.koivunen@aalto.fi}

\author{Robin Rajam\"aki}
\address{Department of Signal Processing and Acoustics, Aalto University, P.O. Box 15400, FI-00076 Aalto, Finland}
\email{robin.rajamaki@aalto.fi}


\begin{document}
\begin{abstract}
  We study a generalization of additive bases into a planar setting.
  A planar additive basis is a set of non-negative integer pairs whose
  vector sumset covers a given rectangle.  Such bases find
  applications in active sensor arrays used in, for example, radar and
  medical imaging.  The problem of minimizing the basis cardinality
  has not been addressed before.

  We propose two algorithms for finding the minimal bases of small
  rectangles: one in the setting where the basis elements can be
  anywhere in the rectangle, and another in the restricted setting,
  where the elements are confined to the lower left quadrant.  We
  present numerical results from such searches, including the minimal
  cardinalities for all rectangles up to $[0,11]\times[0,11]$, and up
  to $[0,46]\times[0,46]$ in the restricted setting.  We also prove
  asymptotic upper and lower bounds on the minimal basis cardinality
  for large rectangles.
\end{abstract}
\maketitle

%%%%%%%%%%%%%%%%%%%%%%%%%%%%%%%%%%%%%%%%%%%%%%%%%%%%%%%%%%%%%%%%%%%%%%
\section{Introduction}
\label{sec:intro}
An additive basis for an interval of integers $[0, n] =
\{0,1,2,\ldots,n\}$ is a set of non-negative integers $A$ such that
$A+A \supseteq [0,n]$.  By extension we define that a \emph{planar
  additive basis} for a rectangle of integers $R=[0,s_x]\times[0,s_y]$
is a set of points with non-negative integer coordinates
\[
A = \{(x_1,y_1),(x_2,y_2),\ldots,(x_k,y_k)\}, \quad\text{such that $A+A \supseteq R$}.
\]
The sumset is defined in terms of vector addition, that is
\[
A+A' = \{(x+x', \; y+y') \suchthat (x,y) \in A, \; (x',y') \in A' \}.
\]

Additive bases for integer intervals have been widely studied since
Rohrbach~\cite{rohrbach1937}.  Often one seeks to maximize~$n$ when
the basis cardinality $|A|=k$ is given.  For small~$k$ this has been
approached with computations
\cite{challis1993,kohonen2014,lunnon1969,riddell1978}, and for
large~$k$ with asymptotic bounds~\cite{kohonen2017,yu2015}.

Less is known about planar additive bases.  Kozick and Kassam
discussed them in an application context, and proposed some simple
designs~\cite{kozick1991}.  In a rather different line of work,
sumsets in vector spaces and abelian groups have been studied with the
interest in how \emph{small} the sumset can be
\cite{eliahou1998,eliahou2005,eliahou2003}.  Boundary effects in
planar sumsets have also been studied~\cite{han2004}.
 
We now aim to minimize the cardinality $k$ of a planar additive basis,
when the target rectangle $R=[0,s_x]\times[0,s_y]$ is given.  To the
best of our knowledge, this combinatorial optimization problem has not
been addressed before.

Planar bases have an application in signal processing, when an array
of sensor elements is deployed on a plane to be used in active
imaging~\cite{naidu2017}.  Here ``active'' means that the sensors both
transmit a signal towards objects such as radar targets or human
tissue, and receive the reflections.  The pairwise vector sums of the
sensor locations make up a virtual sensor array, called the sum
co-array, which may be used to improve imaging resolution
\cite{hoctor1990}.

An important special case is that of \emph{restricted} bases.  A basis
$A$ for $[0,n]$ is restricted if $A \subseteq [0, n/2]$.  Analogously
we define that a basis $A$ for $[0,s_x]\times[0,s_y]$ is restricted if
$A \subseteq [0,s_x/2]\times[0,s_y/2]$.  Apart from practical
motivations related to the physical placing of sensors, with our
algorithms one can minimize $k$ among restricted bases much faster
than among all bases, so larger instances can be solved.  Also,
restricted bases often exhibit interesting structure.

We introduce here the following results.  (1)~A~search algorithm for
finding all bases of a given size for a given rectangle; and the
minimum basis sizes for all rectangles with $s_x,s_y \le 11$.
(2)~A~meet-in-the-middle method that constructs a restricted planar
basis by gluing together four smaller bases, one in each corner; and
the minimum restricted basis sizes for all even $s_x,s_y \le 46$.
(3)~Asymptotic bounds on the minimum basis size for large rectangles.

%%%%%%%%%%%%%%%%%%%%%%%%%%%%%%%%%%%%%%%%%%%%%%%%%%%%%%%%%%%%%%%%%%%%%%
\section{Definitions and preliminary observations}
\label{sec:preliminary}

The target rectangle is $R=[0,s_x]\times[0,s_y]$.  If $R$ is square,
we call it the \emph{$s$-square}, with $s=s_x=s_y$.  A basis
containing $k$ elements is a \emph{$k$-basis}.  The size of the
smallest basis for $[0,s_x]\times[0,s_y]$ is denoted by $k(s_x,s_y)$.

If $s_x$ and $s_y$ are even, we set $h_x = s_x/2$ and $h_y = s_y/2$.
Then a basis $A$ is \emph{restricted} if $A \subseteq [0,h_x] \times
[0,h_y]$. Note that it follows that $A+A=R$. The size of the smallest
restricted basis is $k^*(s_x,s_y)$.

\begin{figure}[b]
  \begin{subfigure}[b]{0.25\textwidth}
    \centerline{\includegraphics[width=1\textwidth]{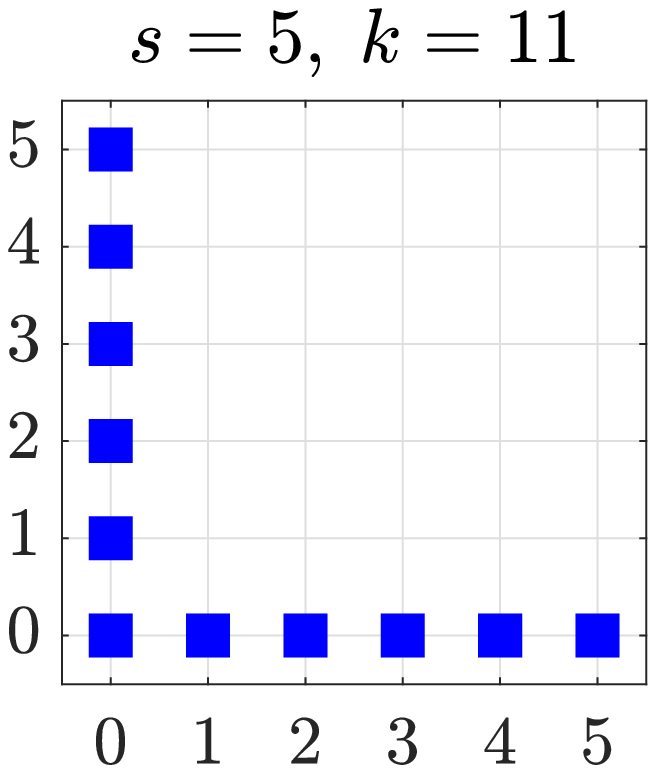}}
    \caption{}
    \label{fig:l_5}
  \end{subfigure}
  \hfill
  \begin{subfigure}[b]{0.25\textwidth}
    \centerline{\includegraphics[width=1\textwidth]{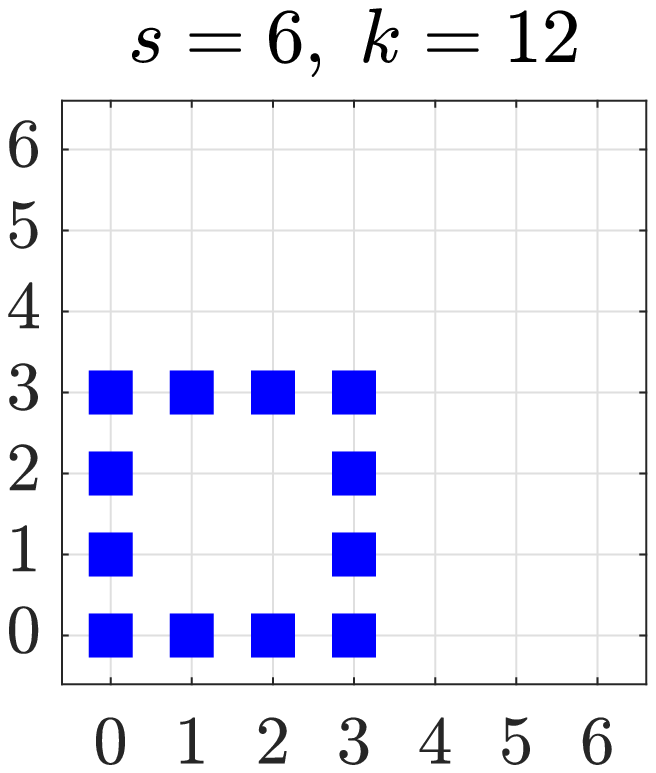}}
    \caption{}
    \label{fig:b_6}
  \end{subfigure}
  \hfill
  \begin{subfigure}[b]{0.25\textwidth}
    \centerline{\includegraphics[width=1\textwidth]{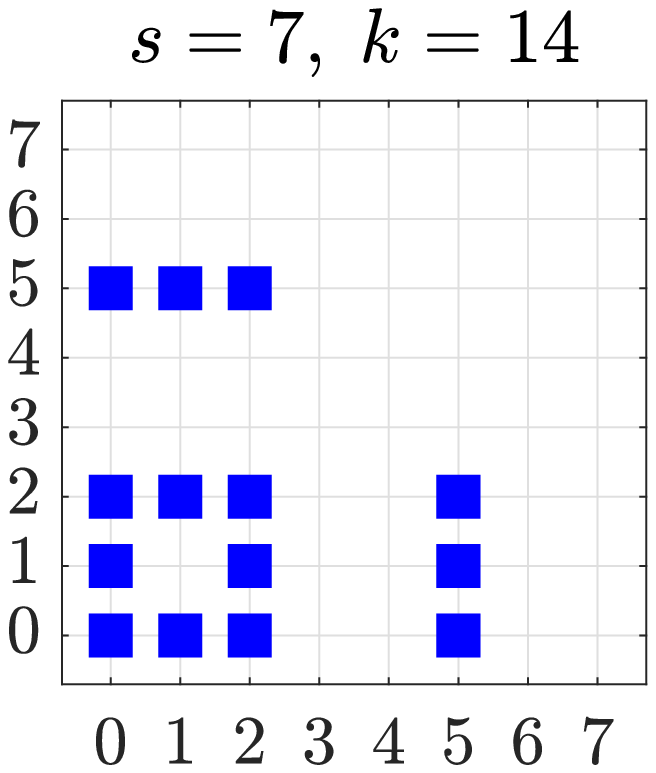}}
    \caption{}
    \label{fig:nr_7}
  \end{subfigure}
  \caption{(a) L-shaped basis for the 5-square. (b) Boundary basis for
    the 6-square. (c) One of the minimal bases for the 7-square.}
  \label{fig:intro}
\end{figure}

Two simple basis constructions were proposed by Kozick and Kassam in
the context of sensor arrays~\cite{kozick1991}.  For any rectangle,
the \emph{L-shaped basis} is
\begin{equation}
  ([0,s_x] \times \{0\}) \;\cup\; (\{0\} \times [0,s_y]),
  \label{eq:lbasis}
\end{equation}
which has $s_x+s_y+1$ elements.  If $s_x,s_y\ge 2$ are even, the
\emph{boundary basis} is
\begin{equation}
  ([0,h_x] \!\times\! \{0,h_y\})
  \; \cup \; (\{0,h_x\} \!\times\! [0, h_y]),
  % ([0,h_x] \!\times\! \{0\})
  % \; \cup \; ([0,h_x] \!\times\! \{h_y\})
  % \; \cup \; (\{0\} \!\times\! [0, h_y])
  % \; \cup \; (\{h_x\} \!\times\! [0, h_y]),
  \label{eq:boundarybasis}
\end{equation}
which has $s_x+s_y$ elements and is restricted.  These two provide a
minimal basis for most small squares (boundary basis if $s\ge 2$ is
even, L-shaped otherwise).  The smallest counterexample is the
7-square, whose minimal bases have only $14$~elements, one less than
the L-shaped basis (see Figure~\ref{fig:nr_7}).  However, for
non-square rectangles, \eqref{eq:lbasis} and \eqref{eq:boundarybasis}
are generally not minimal.  Examples of this will be presented in
Section~\ref{sec:results}, and an asymptotic result in
Section~\ref{sec:largescale}.

If $A$ is a basis for~$R$ such that $A \subseteq R$, we say that $A$
is \emph{admissible}.  If not, then it cannot be minimal, since one
can simply drop the elements that are outside the target.  So we
confine our attention to admissible bases.

The following observations about the corners and the horizontal edges
of planar additive bases will be useful.  Corresponding results in the
vertical direction can be proven by transposing $x$ and~$y$.

\begin{lemma}[Origin corner]
  \label{lemma:corners_nonrestr}
  If $A$ is a basis for a rectangle with $s_x \geq 1$, then $(0,0),
  (1,0) \in A$.
\end{lemma}
\begin{proof}
  The only way to represent $(1,0)$ as a sum of two pairs of
  non-negative integers is $(0,0)+(1,0)$, so those elements must be in
  the basis.
\end{proof}

\begin{lemma}[Restricted edges]
  \label{lemma:edges_restr}
  If $A$ is a restricted basis for $[0,s_x]\times[0,s_y]$, then its
  bottom edge $\{x : (x,0)\in A\}$ and top edge $\{x : (x,h_y)\in A\}$
  are (one-dimensional) restricted bases for $[0,s_x]$.
\end{lemma}
\begin{proof}
  Consider first the bottom edge.  Since the $y$ coordinates in $A$
  are non-negative, for any $x\in[0,s_x]$ the point $(x,0)$ must be
  the sum of some $(x',0),(x'',0) \in A$.  Since $A$ is restricted, we
  have $x',x'' \le h_x$.

  Consider next the top edge.  Since the $y$ coordinates in $A$ are at
  most $h_y$, for any $x\in[0,s_x]$ the point $(x,s_y)$ must be the
  sum of some $(x',h_y),(x'',h_y) \in A$.  Since $A$ is restricted, we
  have $x',x'' \le h_x$.
\end{proof}

\begin{lemma}[Two rows]
  \label{lemma:shallow_restr}
  For any even $s_x \ge 0$, we have $k^*(s_x,2) = 2k^*(s_x,0)$.
\end{lemma}
\begin{proof}
  Let $A$ be a restricted basis for $[0,s_x]\times[0,2]$.  By
  Lemma~\ref{lemma:edges_restr} its bottom and top edges are
  restricted bases for $[0,s_x]$, so each has at least $k^*(s_x,0)$
  elements.  Thus $|A| \ge 2k^*(s_x,0)$.

  To see that $k^*(s_x,2) \le 2k^*(s_x,0)$, let $A^*$ be a restricted
  basis for $[0,s_x]$.  Then $A^* \times [0,1]$ is a restricted
  basis for $[0,s_x]\times[0,2]$.
\end{proof}

%%%%%%%%%%%%%%%%%%%%%%%%%%%%%%%%%%%%%%%%%%%%%%%%%%%%%%%%%%%%%%%%%%%%%%
\section{Search algorithm for admissible bases}
\label{sec:nonrestricted}
Here we develop a method to find all admissible $k$-bases for a given
rectangle.  Then we can also establish the minimum value of~$k$.  For
example, the L-shaped basis suffices to show that $k(9,9) \le 19$, but
to prove that $k(9,9)=19$ we must ascertain that there is no 18-basis
for the $9$-square.  Trying out the $\binom{100}{18} \approx 3 \cdot
10^{19}$ ways of placing 18 elements in $[0,9]\times[0,9]$ is
obviously impractical.

Our Algorithm~\ref{alg:search} is a relatively straightforward
generalization of Challis's algorithm, which finds one-dimensional
bases \cite{challis1993}.  Assume for simplicity that $s_x \ge 2$.
By~Lemma~\ref{lemma:corners_nonrestr} the points $(0,0)$ and $(1,0)$
must be included in the basis.  Next we branch on the decision whether
$(2,0)$ is included.  We proceed to the right and rowwise, branching
at each location on whether that point is included, until we have $k$
elements or reach the top right corner.

\begin{algorithm}[tb]
  \caption{Find all admissible $k$-bases for $[0,s_x]\times[0,s_y]$}
  \label{alg:search}
  \begin{algorithmic}[1]
    \Procedure{FindBases}{$k,s_x,s_y$}
    \State \Call{Extend}{$k,s_x,s_y,\{(0,0),(1,0)\},1,0$}
    \EndProcedure

    \Procedure{Extend}{$k,s_x,s_y,A,x,y$}
    \CommentLine{$(x,y)$ is the latest location considered (either filled or left empty).}
    \State $j \gets \lvert A \rvert$
    \Comment{Number of elements}
    
    \State $G \gets \lvert [0,s_x]\times[0,s_y] \setminus (A+A) \rvert$
    \Comment{Number of gaps}

    \If{$(j = k) \wedge (G=0)$}
    \Call{Print}{$A$}
    \Comment{Found a basis}
    \EndIf

    \If{$j = k$}
    \Return
    \Comment{Reached full size}
    \EndIf

    \State $M \gets (k+j+1)(k-j)/2$
    \Comment{Max. sums to expect}
    \If{$M < G$}\label{line:gaps}
    \Return
    \Comment{Too many gaps}
    \EndIf

    \If{$x<s_x$}
      \State $x \gets x+1$
      \Comment{Proceed right}
    \ElsIf{$y<s_y$}
      \State $x \gets 0$
      \Comment{Begin next row}
      \State $y \gets y+1$
    \Else
      \State \Return
      \Comment{Reached the top right}
    \EndIf

    \If{$(x,y) \in A+A$}\Comment{Already covered?}\label{line:admissibility}
      \State \Call{Extend}{$k,s_x,s_y,A,x,y$}
      \Comment{Branch without $(x,y)$}
    \EndIf
    
    \State \Call{Extend}{$k,s_x,s_y,A \cup \{(x,y)\},x,y$}
    \Comment{Branch with $(x,y)$}
    
    \EndProcedure
  \end{algorithmic}
\end{algorithm}

During the search, two tests prune unfruitful branches.  One of them
(line~\ref{line:admissibility}) concerns unfillable holes in the
sumset.  Suppose that we are currently at $(x,y)$.  Because of the way
how the search proceeds, any location $(x',y')$ considered deeper in
the search will have $x'>x$ or $y'>y$ (or both).  Thus any such
elements will not generate the sum $(x,y)$, by the non-negativity of
coordinates.  If $(x,y)$ has not already been covered, then $(x,y)$
has to be included in the basis.

The other test (line~\ref{line:gaps}) is based on a counting argument.
Suppose that after placing $j$ elements there are $G$ gaps, or target
points not covered by the current sumset.  No matter where the
remaining $k-j$ elements are placed, they will generate at most
$M=(j+1)+(j+2)+\ldots+k = (k+j+1)(k-j)/2$ more sums.  If $M<G$, then
the current search branch cannot lead to any solutions.

This algorithm is quite simple, and there may be several ways to
improve it by exploiting the geometry of the problem.  For example,
instead of proceeding rowwise, the target rectangle can be explored in
a different order: after completing the bottom edge ($y=0$), do next
all of the left edge ($x=0$), then second row, second column, and so
on.  The idea is to introduce necessary conditions from both the left
and bottom edges early~on.  This change does not affect the validity
of the algorithm.  Empirically we observed that it saves about $37\%$
of the running time with 19-bases of the 9-square.

% original measurements:
% k=19 s=9 rowwise 2527 seconds
% k=19 s=9 L-wise  1590 seconds

Typically for a combinatorial branch-and-bound method, the time
requirement of this algorithm grows rapidly as $k$ increases.  We
implemented the algorithm in C++ and ran it on Intel Xeon E7-8890
processors (nominal clock frequency 2.2~GHz).  For 19-bases of the
9-square the search took $0.44$ hours of processor time; for 23-bases
of the 11-square it took $1058$ hours.  Results are summarized in
Table~\ref{tab:results_nonrestr_square} (squares) and
Table~\ref{tab:results_nonrestr_rect} (rectangles).

%%%%%%%%%%%%%%%%%%%%%%%%%%%%%%%%%%%%%%%%%%%%%%%%%%%%%%%%%%%%%%%%%%%%%%
\section{Meet-in-the-middle method for restricted bases} \label{sec:restricted}

In one dimension, i.e.\ for integer intervals, a meet-in-the-middle
(MIM) method to find the optimal restricted bases was proposed by
Kohonen~\cite{kohonen2014meet}.  In its simplest form the method
splits a restricted basis at its midpoint into two components, a
prefix and a suffix, which are then sought separately among the
admissible bases of a smaller interval.  It is much faster to consider
all pairs of these components than to search directly for restricted
bases by a method similar to Algorithm~\ref{alg:search}.  The largest
known optimal restricted bases for integer intervals have been
computed by this method, with $k^*(734,0)=48 $
\cite{kohonen2015early}.

\begin{figure}[tb]
	\centerline{\includegraphics[]{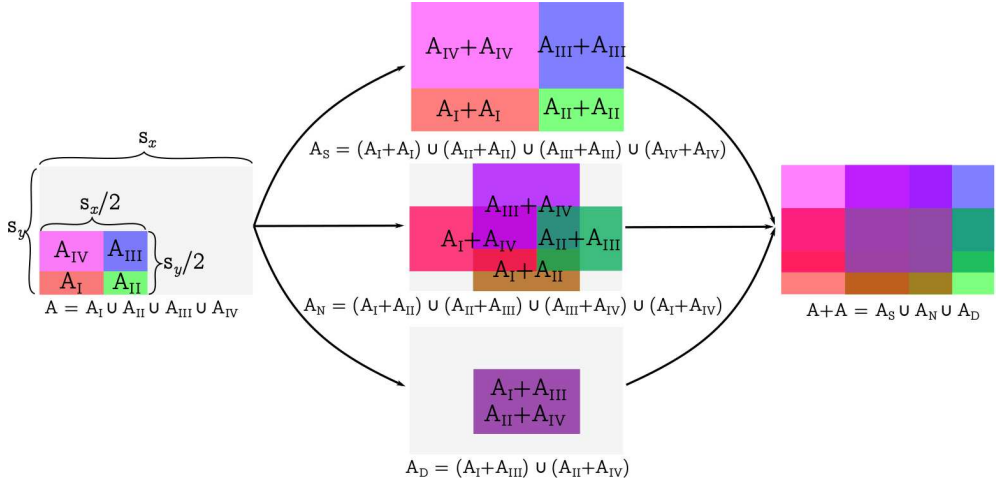}}
	\caption{MIM decomposition of a restricted basis $ A $.  The
          four components $ A_\I,\dots, A_\IV$ are contained in the
          colored rectangles (left).  Consequently, $ A+A $ (right) is
          the union of $A_\text{S} $, $A_\text{N} $ and $A_\text{D} $
          (center), which are the self, neighboring and diagonal sums
          of the components.  The extreme corners of $ A+A $ are
          covered only by the self sumsets, so $ A_\I,\dots, A_\IV $
          must be admissible bases for those rectangles (up to suitable
          coordinate transformations).}
	\label{fig:mim}
\end{figure}

\begin{algorithm}[tb]
  \caption{Find all restricted $k$-bases for $[0,s_x]\times[0,s_y]$}
  \label{alg:mim}
  \begin{algorithmic}[1]
    \Procedure{MIM}{$k,s_x,s_y$}
      \State $ h_x \gets s_x/2 $ \Comment{dimensions of rectangle containing $ A$}
      \State $ h_y \gets s_y/2 $
      \State $ a_x \gets \lfloor h_x/2 \rfloor $ \Comment{dimensions of rectangle containing $ A_\I $}
      \State $ a_y \gets \lfloor h_y/2 \rfloor $
      \State $ b_x \gets h_x-a_x-1 $             \Comment{dimensions of other rectangles}
      \State $ b_y \gets h_y-a_y-1 $
      \State $ \kmin_\I   \gets k(a_x,a_y)$		\Comment{look up minimum sizes of the components}
      \State $ \kmin_\II  \gets k(b_x,a_y)$
      \State $ \kmin_\III \gets k(b_x,b_y)$
      \State $ \kmin_\IV  \gets k(a_x,b_y)$
      \CommentLine{Iterate feasible ways of allocating $k$ among the four quadrants}
      \For{$(k_\I,k_\II,k_\III,k_\IV)$ such that $k_\I+k_\II+k_\III+k_\IV=k$}
        \If {$k_\I\ge \kmin_\I \;\wedge\; k_\II\ge \kmin_\II \;\wedge\; k_\III\ge \kmin_\III \;\wedge\; k_\IV\ge \kmin_\IV$}
          \CommentLine{Compute or look up admissible component bases}
          \State $\Bcal_\I \gets$  output from \Call{FindBases}{$k_\I,a_x,a_y$} 
          \State $\Bcal_\II \gets$ output from \Call{FindBases}{$k_\II,b_x,a_y$}
          \State $\Bcal_\III\gets$ output from \Call{FindBases}{$k_\III,b_x,b_y$}
          \State $\Bcal_\IV \gets$ output from \Call{FindBases}{$k_\IV,a_x,b_y$}
          \For{$ (B_\I, B_\II, B_\III, B_\IV) \in \Bcal_\I \times \Bcal_\II \times \Bcal_\III \times \Bcal_\IV $}\label{line:innerloop}
            \State $A_\I   \gets B_\I$
            \State $A_\II  \gets \{(h_x\!-\!x, \; y) : (x,y)\in B_\II\}$ \Comment{Mirror $x$ coordinates}
            \State $A_\III  \gets \{(h_x\!-\!x, \; h_y\!-\!y) : (x,y)\in B_\III\}$ \Comment{Mirror $x,y$ coordinates}
            \State $A_\IV  \gets \{(x, \; h_y\!-\!y) : (x,y)\in B_\IV\}$ \Comment{Mirror $y$ coordinates}
            \State $A \gets A_\I \cup A_\II \cup A_\III \cup A_\IV$ \Comment{Glue components}\label{line:glue}
            \If {$ A+A = R $} \Call{Print}{$A$} \Comment{Found a basis} \EndIf
          \EndFor
        \EndIf       
    \EndFor
    \EndProcedure
  \end{algorithmic}
\end{algorithm}

Here the MIM method is extended to the planar setting.  We want to
find all $k$-bases for $R=[0,s_x]\times[0,s_y] $, subject to the
restriction $A \subseteq R_h = [0,h_x]\times[0,h_y]$, where
$h_x=s_x/2>0$ and $h_y=s_y/2>0$.  First divide $R_h$ into four
disjoint rectangles as follows. Choose breaking points
$a_x \in [0,h_x-1]$ and $a_y \in [0,h_y-1]$ arbitrarily, and define
\begin{align*}
  R_\I   &= [0,a_x]\times[0,a_y], \\
  R_\II  &= [a_x+1,h_x]\times[0,a_y], \\
  R_\III &= [a_x+1,h_x]\times[a_y+1,h_y], \\
  R_\IV  &= [0,a_x]\times[a_y+1,h_y].
\end{align*}
These are the colored rectangles in Figure~\ref{fig:mim} (left).  Now
split a basis $A$ into components $ A_\I,A_\II,A_\III,A_\IV$ so that
$ A_\I = A \cap R_\I$, and similarly with the others.  By the
non-negativity of all coordinates, any sumset involving $A_\II$,
$A_\III$ or $A_\IV$ is completely outside the lower left corner
~$R_\I$.  So in order to have $A+A \supseteq R$ we need
$A_\I+A_\I \supseteq R_\I$.  That is, $A_\I$ must be an admissible
basis for~$R_\I$.  All candidates for~$A_\I$ can be listed by
Algorithm~\ref{alg:search}.

A similar argument applies in the lower right corner of the target,
with some necessary coordinate transformations.  Let
$C_\II = [h_x+a_x+1, s_x] \times [0,a_y]$.  Then we need
$A_\II+A_\II \supseteq C_\II$, since all the other component sumsets
are outside~$C_\II$.  Consider the ``mirror image'' of $A_\II$, namely
$B_\II = \{(h_x\!-\!x,\; y) : (x,y) \in A_\II\}$.  By construction, we
have $B_\II \subseteq [0,b_x]\times[0,a_y]$, where for convenience we
have written $b_x = h_x-a_x-1$.  Now the condition
$A_\II+A_\II \supseteq C_\II$ implies that
$B_\II+B_\II \supseteq [0,b_x]\times[0,a_y]$.  So $B_\II$ must be an
admissible basis for~$[0,b_x]\times[0,a_y]$, and again all candidates
can be found by Algorithm~\ref{alg:search}.

Similar conditions for $A_\III$ and~$A_\IV$ apply in the remaining two
corners.  Consequently, $A$ must be the union of four components,
which are (mirror images of) admissible bases of suitable rectangles.
Since we have so far only dealt with necessary conditions, we have not
lost any possible solutions.  The conditions guarantee only that the
four extreme corner regions are covered; for any candidate solution
$A = A_\I \cup A_\II \cup A_\III \cup A_\IV$ we must finally check
whether in fact $A+A \supseteq R$.

Algorithm~\ref{alg:mim} gives a formal description of the MIM method.
We choose $ a_x = \lfloor h_x/2 \rfloor$ and
$a_y = \lfloor h_y/2 \rfloor$ so the components have roughly equal
dimensions.  The final ingredient of the algorithm, on lines 8--14,
concerns how the overall budget of $k$ elements is allocated to the
four components.  Note that $A_\I$ need not be a minimal basis
for~$R_\I$.  It may have more than $k(a_x,a_y)$ elements, and indeed
this may be necessary to find any solutions for $A+A \supseteq R$.
The same goes for the other three components.

In order to determine the value of $k^*(s_x,s_y)$, just run
Algorithm~\ref{alg:mim} repeatedly, beginning with
$k=\kmin_\I+\kmin_\II+\kmin_\III+\kmin_\IV$ since certainly there are
no solutions below that size, and increase $k$ in steps of~$1$ until
some solutions are found.

\begin{example}
  A restricted basis $ A $ for $ R = [0,10]\times[0,10] $ satisfies
  $A\subseteq R_h=[0,5]\times[0,5]$.  The first quadrant of $R_h$ is
  $R_\I = [0,2]\times[0,2]$, and the other quadrants have the same
  size.  Since $ k(2,2) =4$, we have necessarily
  $|A| \geq 4+4+4+4 = 16 $.  There is only one $4$-basis for
  $[0,2]\times[0,2]$, so for $k=16$ there is only one combination to
  check in the innermost loop of Algorithm~\ref{alg:mim}.  But this
  combination does not give a basis for $[0,10]\times[0,10]$, so more
  than $16$ elements are needed.

  It turns out that $k=20$ is enough.  After some simple pruning
  conditions (not shown in Algorithm~\ref{alg:mim}) we find that the
  only possible allocations of $20$ elements are
  $(k_\I,k_\II,k_\III,k_\IV) = (4,6,4,6)$ and $(5,5,5,5)$.  There are
  nine $5$-bases and eighteen $6$-bases for $[0,2]\times[0,2]$, so the
  first allocation leads to $1 \cdot 18 \cdot 1 \cdot 18 = 324$
  combinations to be checked, and the second gives
  $9 \cdot 9 \cdot 9 \cdot 9 = 6561$ combinations.  Out of these, we
  find $17$ restricted solutions.  This is less than one second of
  computation.  In comparison, finding \emph{all} $20$-bases for the
  $10$-square with our implementation of Algorithm~\ref{alg:search}
  takes more than an hour.
\end{example}

There are a few ways to significantly prune the number of candidate
solutions that need to be checked.  Firstly, the complete sumset of a
candidate restricted basis does not have to be calculated immediately.
A necessary condition for a restricted basis is that any two
neighboring quadrants form a restricted basis along one of the
coordinate axes. It therefore suffices to first check whether this
condition is satisfied for all four neighboring quadrant pairs. Only
if the condition is met, then the full sumset needs be checked.

Secondly, often some of the component pieces have the same dimensions
(indeed all of them if $h_x,h_y$ are both odd).  If the pieces also
have the same cardinality, then the set of candidate solutions is the
same for both of them, up to suitable coordinate transformations.

\begin{example}
  Consider a restricted basis $ A $ for the square $[0,s]\times[0,s]$,
  with $s/2 = 2a+1$ odd and $a \ge 0$.  Each quadrant has the same
  dimensions $a_x=b_x=a_y=b_y=a$.  If all component sets also have
  equal cardinality, then the candidates for $A_\II$, $A_\III$ and
  $A_\IV$ are the same as for $A_\I$, up to suitable mirroring.
  Furthermore, if the sumset $ (A_\I \cup A_\II) + (A_\I \cup A_\II)$
  does not cover $[0,s]\times[0,a]$, then all candidate solutions
  containing any rotation of this pair can be pruned.
\end{example}

Thirdly, when components have different cardinalities, the order in
which they are glued matters.  One possible strategy is to first glue
component pairs of low cardinality, not only because they usually have
fewer component solutions to glue, but also because they are less
likely to produce possible gluings than pairs of higher
cardinality. Occasionally, a pairwise gluing that has no solutions
rules out all combinations containing high cardinality components.
Then these components do not even have to be computed in the first
place.

\begin{example}
  Let the cardinality of a square restricted basis be
  $k^* + \tilde{k}^*= 4\kmin+
  (\tilde{k}_\I+\tilde{k}_\II+\tilde{k}_\III+\tilde{k}_\IV) $,
  where $\tilde{k}_\I,\dots, \tilde{k}_\IV $ represent the number of
  extra elements in each quadrant. If $ \tilde{k}^*=3 $, then there
  are four ways to distribute the extra element:
  $(\tilde{k}_\I,\tilde{k}_\II,\tilde{k}_\III,\tilde{k}_\IV) =
  (0,0,0,3)$,
  $(0,0,1,2)$, $(0,1,0,2)$, or $(0,1,1,1)$. If the gluing with
  $(\tilde{k}_\I,\tilde{k}_\II) = (0,0) $ gives no solutions, then the
  candidate solutions containing pairs
  $(\tilde{k}_\III,\tilde{k}_\IV) = (0,3) $ and
  $(\tilde{k}_\III,\tilde{k}_\IV) = (1,2) $ are discarded. More
  importantly, solutions for the $ (\kmin + 3 )$-basis do not have to
  be computed at all.
\end{example}

%%%%%%%%%%%%%%%%%%%%%%%%%%%%%%%%%%%%%%%%%%%%%%%%%%%%%%%%%%%%%%%%%%%%%%
\section{Numerical results}
\label{sec:results}

We now describe some results obtained for small rectangles with
Algorithms \ref{alg:search} and \ref{alg:mim}.  Examples of minimal
bases are shown in Figures~\ref{fig:nr_bases} and \ref{fig:r_bases}.
We note that especially the restricted solutions in
Figure~\ref{fig:r_bases} exhibit regular structure that can perhaps be
generalized to larger bases.

In the result listings, $m$ is the number of all minimal bases, and
$ m_\text{u} $ is the number of ``unique'' bases after taking into
account rotation and mirror symmetries.  Each basis may have up to 8
symmetric variants if the target is square, and up to 4 variants
otherwise.
  
\begin{figure}[b]
  \centerline{\includegraphics[width=1\textwidth]{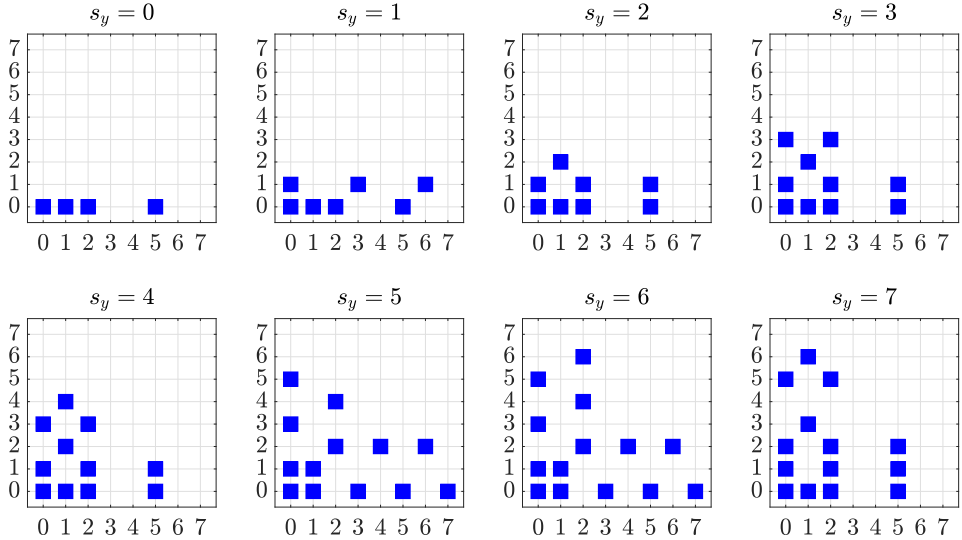}}
  \caption{Some minimal bases for $s_x=7$ and varying $s_y$.}
  \label{fig:nr_bases}
\end{figure}

\begin{figure}[t]
  \centerline{\includegraphics[width=1\textwidth]{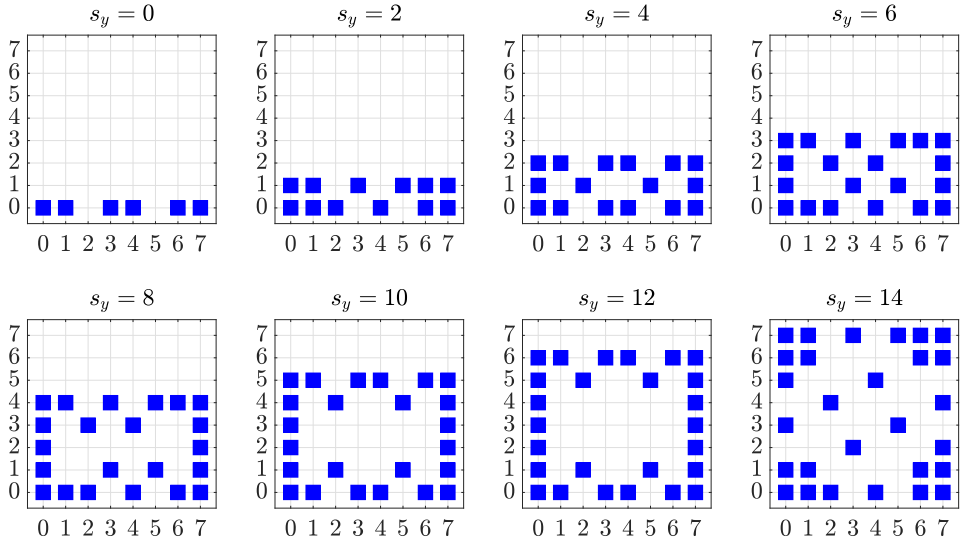}}
  \caption{Some minimal restricted bases for $ s_x = 14$, varying $s_y$.}
  \label{fig:r_bases}     
\end{figure}

\subsection{Results for squares}
Table~\ref{tab:results_nonrestr_square} summarizes the minimal bases
for squares up to $ s=11 $.  We observe that in the even-sided
instances $s=2,4,6,8,10$ one of the minimal solutions is the boundary
basis.  In the odd-sided instances $s=1,3,5,9,11$ one of the minimal
solutions is the L-shaped basis.  The case $s=7$ stands out as an
exception where the L-shaped basis is not minimal (see also
Figure~\ref{fig:nr_7}).  One may observe that when $s$ is even, the
number of minimal bases is relatively small.  This may be understood
as their cardinality is only $2s$, while in the odd cases the
cardinality is usually $2s+1$.

Table~\ref{tab:results_restr_square} summarizes the minimal restricted
bases for squares up to $s=46$.  For $s\le 26$ we generated and
counted the minimal bases.  For $28 \le s \le 46$ we only determined
the value of $k^*(s,s)$, but did not generate the bases.  For example,
since we found that there is no restricted $91$-basis for the
$46$-square, we can deduce that $k^*(46,46)=92$ as the boundary basis
has this size.  In all even-sided squares with $2 \le s \le 46$, we
have $k^*(s,s)=2s$, which is attained by the boundary basis.

Although the simple L-shaped and boundary bases provide minimal or
almost minimal solutions for small squares, having the full collection
of minimal solutions can be useful from an application perspective.
In some sensor array applications it is beneficial to avoid placing
sensor elements near each other, so as to avoid mutual coupling
effects that cause degraded performance~\cite{liu2017}.  This may lead
to a secondary optimization goal, and one may search the collection of
minimal-size bases in order to optimize for this goal.

\subsection{Results for rectangles}

The situation with rectangles is quite different from that with
squares: if the aspect ratio $\rho=(s_y+1)/(s_x+1)$ is far enough
from~1, then minimal bases may be much smaller than the
L-shaped and boundary bases.

Minimal bases for rectangles are summarized in
Table~\ref{tab:results_nonrestr_rect}, and Tables
\ref{tab:results_restr_rect} and~\ref{tab:results_restr_sy2} for the
restricted case.  In order to compare the minimal solutions to the
L-shaped and boundary bases, the quantity $ \Delta k = k-k_\text{t} $
is computed. Here $ k_\text{t} $ is the number of elements in the best
applicable trivial solution, which is the boundary basis when $s_x$
and $s_y$ are even, and the L-shaped basis otherwise, except when
$ s_y = 0 $ where the trivial solution is a one-dimensional basis with
$\lceil s_x/2 \rceil +1 $ elements.

In general, minimal bases use increasingly fewer elements than the
trivial solutions as the aspect ratio deviates further from~1. This is
apparent from Figure~\ref{fig:rho_vs_krel}, which shows the ratio
$ k/k_t $ for minimal bases as a function of aspect ratio.  We observe
a similar behavior for minimal restricted bases in
Table~\ref{tab:results_restr_rect}. In fact a kind of threshold seems
to exist near $s_y \approx s_x/2$, such that below this threshold the
minimal solutions are smaller than trivial, and there are few of them.
Above the threshold the minimal solutions match the trivial, and there
are many of them.  We have currently no explanation for such a
threshold nor for its exact location.

Another peculiarity is illustrated in Figure~\ref{fig:r_sy2}, which
shows two minimal restricted bases for which the number of elements
actually decreases as the target width increases.  Not only is
$k^*(62,2) = 28 > k^*(64,2) = 26$, but the number of solutions for the
two cases is also drastically different.  The former has $125247$
unique solutions, whereas the latter has only~$1$.  The solutions for
$s_y=2 $ listed in Table~\ref{tab:results_restr_sy2} reveal that a
similar effect also occurs for $ s_x=104 $ and $ 116 $.  The same also
applies to $ s_y=0 $, since $ k^*(s_x,2) = 2k^*(s_x,0) $ by
Lemma~\ref{lemma:shallow_restr}.

An overview of currently known minimal restricted bases is shown in
Figure~\ref{fig:r_map}.  The colors of the pixels correspond to the
minimal number of elements.  At the present, bases up to about $k=50$
are practical to list exhaustively.  For clarity of presentation,
restricted one-dimensional bases are not considered here for
$ s_x>120 $.

\begin{figure}[!b]
  \centerline{\includegraphics[width=1\textwidth]{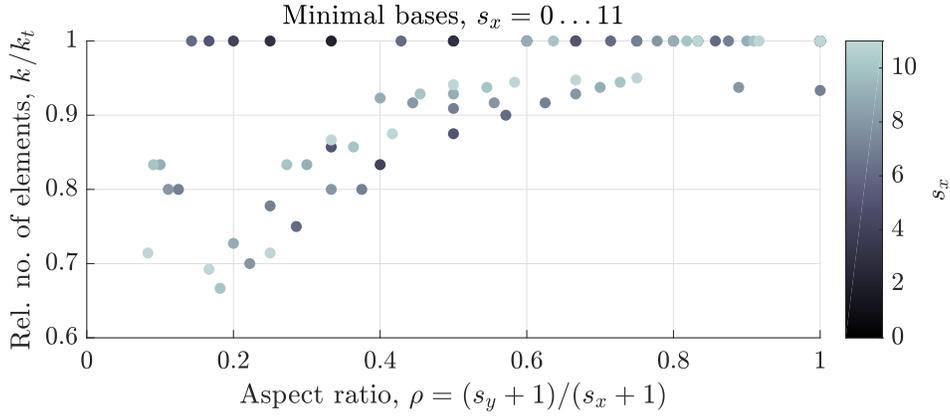}}
  \caption{Number of elements in minimal bases w.r.t. trivial
    bases. Trivial solutions are not optimal for low aspect ratios.}
    \label{fig:rho_vs_krel}
\end{figure}

\begin{figure}[!b]
\begin{subfigure}{1\textwidth}
        \centerline{\includegraphics[width=1\textwidth]{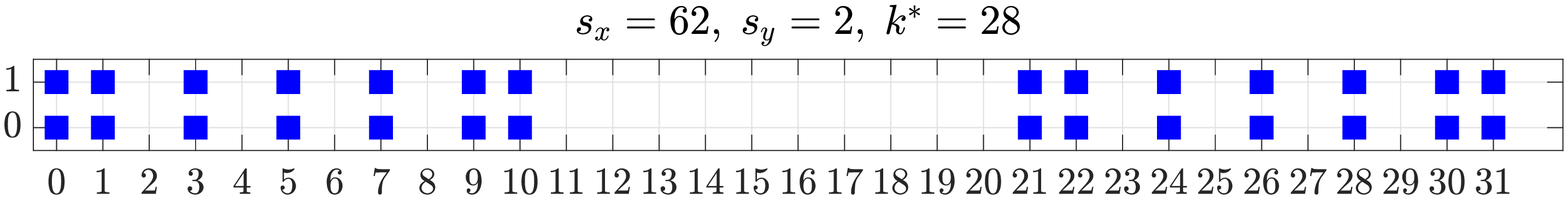}}
        \label{fig:r_sy2_a}
    \end{subfigure}
    	\begin{subfigure}[b]{1\textwidth}
        \centerline{\includegraphics[width=1\textwidth]{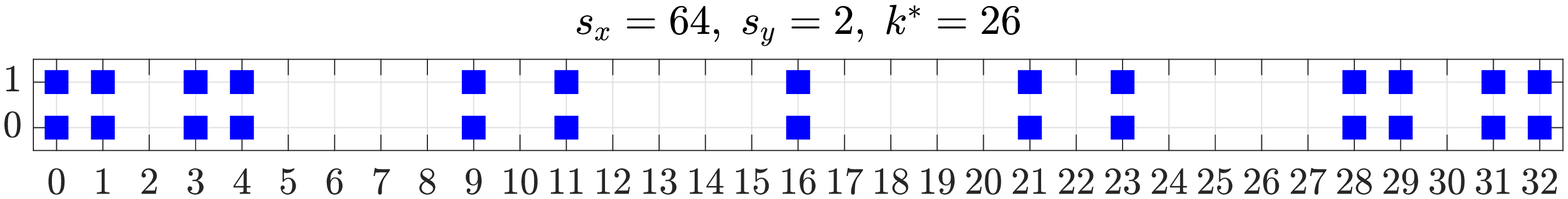}}
       \label{fig:r_sy2_b}
    \end{subfigure}
    \caption{Two restricted bases for $ s_y=2 $, for which the minimal
      number of elements decreases as the rectangle width increases.}
        \label{fig:r_sy2}
\end{figure}

\begin{figure}[!b]
        \centerline{\includegraphics[width=1\textwidth]{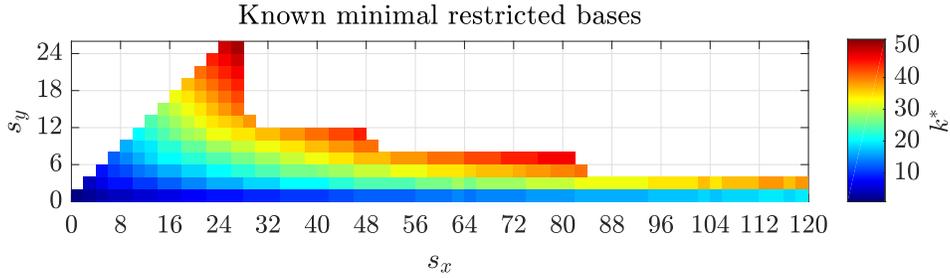}}
	\caption{Minimal number of elements in restricted bases.}
        \label{fig:r_map}
\end{figure}

%%%%%%%%%%%%%%%%%%%%%%%%%%%%%%%%%%%%%%%%%%%%%%%%%%%%%%%%%%%%%%%%%%%%%%
\section{Bounds for large-scale behaviour}
\label{sec:largescale}

For very large rectangles it seems difficult to determine the minimum
basis size exactly.  Towards understanding the large-scale behaviour
we can offer some upper and lower bounds.  We relate the basis size
$k=|A|$ to the number of target elements
$N = \left\lvert\,[0,s_x]\times[0,s_y]\,\right\rvert =
(s_x+1)(s_y+1)$,
which may be understood as the \emph{target area} measured in grid
points.  The \emph{efficiency} of a basis is defined as
\[
c = N/k^2.
\]
The shape of the target is characterized by its \emph{aspect ratio}
$\rho = (s_y+1)/(s_x+1)$.

\subsection{Upper bounds}

A crude upper bound on efficiency is obtained by observing that from
$k$ elements at most $(k+1)k/2$ different pairwise sums can be formed,
considering that $a+b=b+a$ and that sums of the form $a+a$ are
allowed.  It follows that $N \le (k+1)k/2$, so for any planar basis we
have
\begin{equation*}
  c \le 0.5 + O\bigl(1/\sqrt{N}\bigr).
  \label{eq:upperany}
\end{equation*}
In one dimension, upper bounds tighter than $0.5$ have been
established by analytic and combinatorial methods.  For all $s_x$
large enough, by Yu's Theorem~1.1 in~\cite{yu2015} we have
\begin{equation}
  s_x / k(s_x,0)^2 \le 0.45851 = \alpha,
  \label{eq:yu_nonrestr}
\end{equation}
and by Yu's Theorem~1.2 in~\cite{yu2009} we have
\begin{equation}
  s_x / k^*(s_x,0)^2 \le 0.41983 = \beta.
  \label{eq:yu_restr}
\end{equation}
Combining Yu's theorems with simple counting, we obtain the following
bounds with rectangles of small constant height.  For brevity, if $P$
is a set of points, we denote $P_y = \{x : (x,y) \in P\}$ and call
this the \emph{row $y$} of~$P$.

\begin{theorem}
  For all $s_x$ large enough, any basis for $[0,s_x]\times[0,1]$ has
  efficiency $c < 0.4311$.
  \label{thm:upper1}
\end{theorem}
\begin{proof}
  Assume that $s_x$ is large enough that \eqref{eq:yu_nonrestr} holds.
  Without loss of generality let $A$ be admissible, and
  let its rows $A_0,A_1$ contain $k_0,k_1$ elements, respectively.
  Now $A_0+A_0$ must cover $R_0 = [0,s_x]$, and $A_0+A_1$ must cover
  $R_1 = [0,s_x]$.  By applying \eqref{eq:yu_nonrestr} on row~$0$, and
  by counting sums on row~$1$, we obtain
  \begin{align*}
    s_x &\le \alpha k_0^2,\\
    s_x &\le k_0 k_1.
  \end{align*}
  For any $k$, the minimum of these two bounds is maximized at $k_1 =
  \alpha k_0$, implying that $k = (1+\alpha)k_0$ and
  \[
  s_x/k^2 \le \frac{\alpha}{(1+\alpha)^2} < 0.215542.
  \]
  Since $N = |R| = 2(s_x+1)$, we have $N/k^2 < 0.4311$ for $s_x$ large
  enough.
\end{proof}

\begin{theorem}
  For all $s_x$ large enough, any basis for $[0,s_x]\times[0,2]$ has
  efficiency $c < 0.4190$.
  \label{thm:upper2}
\end{theorem}
\begin{proof}
  Assume that $s_x$ is large enough that \eqref{eq:yu_nonrestr} holds.
  Without loss of generality let $A$ be admissible, and let its rows
  $A_0,A_1,A_2$ contain $k_0,k_1,k_2$ elements, respectively.  Now
  $A_0+A_0$ must cover $R_0 = [0,s_x]$, and $A_0+A_1$ must cover $R_1
  = [0,s_x]$, and finally $(A_0+A_2) \cup (A_1+A_1)$ must cover $R_2 =
  [0,s_x]$.  By applying \eqref{eq:yu_nonrestr} on row~$0$, and by
  counting sums on rows $1$ and~$2$, we obtain
  \begin{align*}
    s_x &\le \alpha k_0^2,\\
    s_x &\le k_0 k_1,\\
    s_x &\le k_0 k_2 + k_1^2/2 + k_1/2.
  \end{align*}
  For any $k$, the minimum of these three bounds is maximized at their
  intersection, and by routine manipulations we obtain
  \[
  s_x/k^2 \le \frac{\alpha}{(1+2\alpha-\alpha^2/2)^2} < 0.139663
  \]
  for $s_x$ large enough.  Since $N = |R| = 3(s_x+1)$, we have $N/k^2
  < 0.4190$ for $s_x$~large enough.
\end{proof}

Any improvements to the one-dimensional bound~\eqref{eq:yu_nonrestr}
will imply corresponding improvements to Theorems \ref{thm:upper1}
and~\ref{thm:upper2}.  One could also apply the same proof technique
with larger constant values of $s_y$, but it then becomes more
complicated to maximize the simultaneous upper bounds of $s_x$.
Numerical maximization suggests decreasing upper bounds as $s_y$
increases, for example, around $0.4126$ with $s_y=3$, and around
$0.4087$ with $s_y=4$.  This begs the question: what happens when $s_y$ goes to infinity?

Turning our attention to the restricted case we obtain the following
bounds.

\begin{theorem}
  For all $s_x$ large enough, any restricted basis for $[0,s_x]\times[0,2]$
  has efficiency $c < 0.3149$.
\end{theorem}
\begin{proof}
  Combine Lemma~\ref{lemma:shallow_restr} with the
  bound~\eqref{eq:yu_restr} and the fact that $|R|=3(s_x+1)$.
\end{proof}

\begin{theorem}
  For all $s_x$ large enough, any restricted basis for $[0,s_x]\times[0,4]$
  has efficiency $c < 0.3585$.
\end{theorem}
\begin{proof}
  Assume $s_x$ is large enough that \eqref{eq:yu_restr} holds.  Let
  $A$ be a restricted basis for~$R$, and let $k_0,k_1,k_2$ be the
  cardinalities of its rows.  By applying \eqref{eq:yu_restr} on rows
  $0$ and~$4$ of the target, and by counting sums on rows $1$ and~$3$,
  we obtain
  \begin{align*}
    s_x &\le \beta k_0^2,\\
    s_x &\le k_0 k_1,\\
    s_x &\le k_1 k_2,\\
    s_x &\le \beta k_2^2.
  \end{align*}
  The minimum of these four bounds is maximized at their intersection,
  where $k_0=k_2$ and $k_1=\beta k_0$, thus $k = (2+\beta)k_0$.
  Then we obtain
  \[
  s_x/k^2 \le \frac{\beta}{(2+\beta)^2} < 0.071698.
  \]
  Since $N = |R| = 5(s_x+1)$, we have $N/k^2 < 0.3585$ for $s_x$ large
  enough.
\end{proof}

% As a final remark, we have observed from our numerical results in
% section~\ref{sec:results} that the efficiency of minimal bases in
% general tends towards $c\approx 1/4 $ with increasing target
% size. However as shown above, this value is not necessarily reached
% for all aspect ratios.

\subsection{Lower bounds}

As with one-dimensional bases, also in planar bases it is relatively
easy to obtain an efficiency of approximately $1/4$ for large
rectangles.  For squares this is particularly easy: the L-shaped basis
for an $s$-square has $k=2s+1$, so $c = 0.25 + O(1/s)$.  The boundary
basis has $k=2s$, so its efficiency has the same asymptotic form.

For non-square rectangles, however, the L-shaped and boundary bases
are asymptotically suboptimal.  Consider rectangles
$[0,s_x]\times[0,s_y]$ with a constant aspect ratio $\rho \ne 1$.  The
L-shaped basis has $k=s_x+s_y+1 = (1+\rho)s_x+\rho$, so
\[
c \to \rho / (1+\rho)^2 < 1/4
\]
as~$s_x \to \infty$.  The case with the boundary basis is similar.
For example, if the aspect ratio is $\rho=9$, then both the L-shaped
and boundary bases have only $c\to 0.09$ in the limit.

When $\rho \ne 1$, one may prefer one of the following two parametric
constructions that achieve an asymptotic efficiency of~$1/4$.  Both
constructions are illustrated in Figure~\ref{fig:quarter}.  We use
here the notation
\[
[a,(t),b] = \{a,\; a+t,\; a+2t,\; \ldots,\; b\}
\]
for a finite arithmetic progression from $a$ to $b$ with step
length~$t$, with the provision that $b-a$ is divisible by~$t$.

\begin{figure}[b]
  \begin{subfigure}[b]{0.48\textwidth}
    \centerline{\includegraphics[width=1\textwidth]{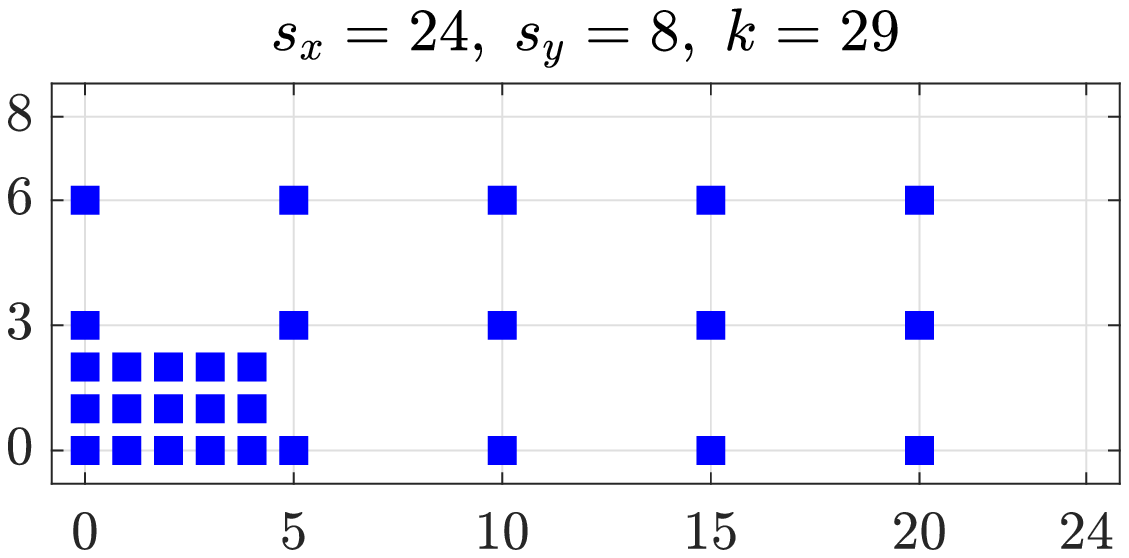}}
    \caption{}
    \label{fig:quarter_ds}
  \end{subfigure}
  \hfill
  \begin{subfigure}[b]{0.48\textwidth}
    \centerline{\includegraphics[width=1\textwidth]{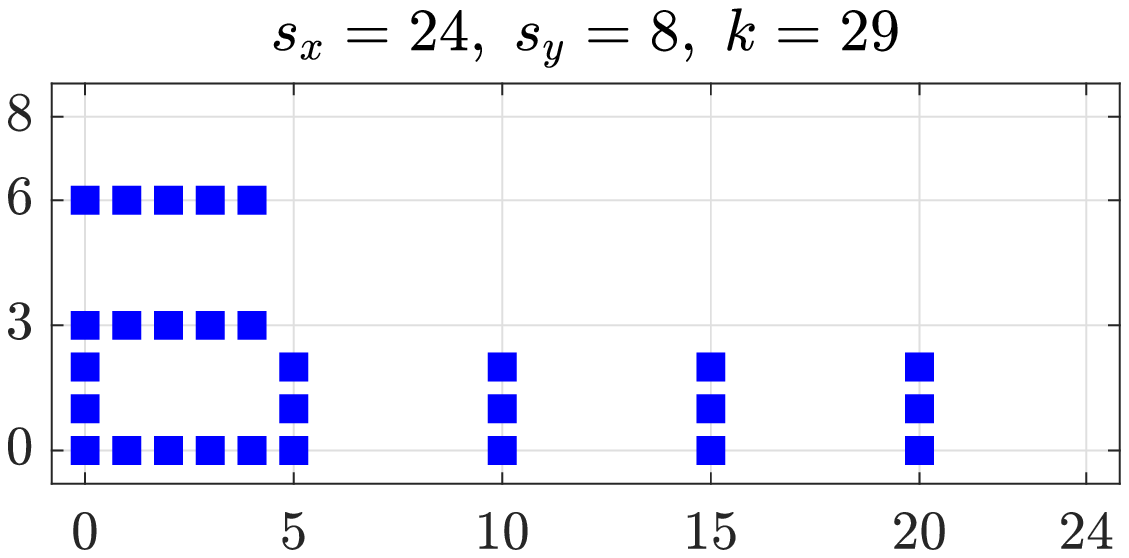}}
    \caption{}
    \label{fig:quarter_sb}
  \end{subfigure}
  \caption{Two basis constructions for rectangles: (a) a dense-sparse
    basis, (b) a short-bars basis, both with parameters $t_x=5$,
    $t_y=3$.  Both have only $29$ elements while an L-shaped basis for
    the same rectangle would have $24+8+1=33$.}
    \label{fig:quarter}
\end{figure}

\begin{definition}
  \label{def:densesparse}
  The \emph{dense-sparse basis} with parameters $t_x,t_y \ge 1$ is the
  set $A = B \cup C$, where $B = [0,t_x-1] \times [0,t_y-1]$ and $C =
  [0,(t_x),t_x^2-t_x] \times [0,(t_y),t_y^2-t_y]$.
\end{definition}

\begin{theorem}
  \label{thm:densesparse}
  The dense-sparse basis has $|A|=2t_xt_y-1$ and $A+A \supseteq
  [0,t_x^2-1]\times[0,t_y^2-1]$.
\end{theorem}
\begin{proof}
  Since $|B|=|C|=t_xt_y$ and $B \cap C = \{(0,0)\}$, the claim on
  $|A|$ follows.  For any point $(x,y) \in R$, let $x = b_x + c_x$
  with $b_x \in [0,t_x-1]$ and $c_x \in [0,(t_x),t_x^2-t_x]$.
  Similarly let $y = b_y + c_y$ with $b_y \in [0,t_y-1]$ and $c_y \in
  [0,(t_y),t_y^2-t_y]$.  Now $(x,y) = (b_x,b_y) + (c_x,c_y)$ with
  $(b_x,b_y)\in B$ and $(c_x,c_y) \in C$.  Thus $(x,y) \in B+C
  \subseteq A+A$.
\end{proof}

\begin{definition}
  \label{def:shortbars}
  The \emph{short-bars basis} with parameters $t_x,t_y \ge 1$ is the
  set $A = B \cup C$, where $B = [0,t_x-1] \times [0,(t_y),t_y^2-t_y]$
  and $C = [0,(t_x),t_x^2-t_x] \times [0,t_y-1]$.
\end{definition}

\begin{theorem}
  \label{thm:shortbars}
  The short-bars basis has $|A|=2t_xt_y-1$ and $A+A \supseteq
  [0,t_x^2-1]\times[0,t_y^2-1]$.
\end{theorem}
\begin{proof}
  Since $|B|=|C|=t_xt_y$ and $B \cap C = \{(0,0)\}$, the claim on $|A|$
  follows.  For any point $(x,y) \in R$, let $x = b_x + c_x$ with
  $b_x \in [0,t_x-1]$ and $c_x \in [0,(t_x),t_x^2-t_x]$.  Similarly
  let $y = b_y + c_y$ with $b_y \in [0,(t_y),t_y^2-t_y]$ and
  $c_y \in [0,t_y-1]$.  Now $(x,y) = (b_x,b_y) + (c_x,c_y)$ with
  $(b_x,b_y)\in B$ and $(c_x,c_y) \in C$.  Thus
  $(x,y) \in B+C \subseteq A+A$.
\end{proof}

\begin{corollary}
  Let $\rho = p^2/q^2$ be a fixed aspect ratio, where $p$ and $q$ are
  integers, and let $h \ge 1$ be an integer.  Then both the
  dense-sparse basis and the short-bars basis, with parameters
  $t_x = qh$ and $t_y = ph$, are bases for the rectangle
  $[0,t_x^2-1] \times [0, t_y^2-1]$, which has the said aspect ratio.
  The efficiency of either basis is
  \[
  c = \frac{t_x^2 t_y^2}{(2 t_x t_y - 1)^2} = 0.25 + O(1/h^2).
  \]
\end{corollary}

For arbitrarily wide rectangles of any \emph{constant height} we
present a basis construction whose asymptotic efficiency
\emph{exceeds} $1/4$.  The construction is somewhat analogous to
Mrose's one-dimensional basis~\cite{mrose1979}, hence the name.

\begin{definition}
  \label{def:stacked}
  The \emph{stacked Mrose basis} with parameters $s_y \ge 0$ and $t \ge 1$
  is the set $I_1 \cup I_2 \cup I_3 \cup T \cup S$, where
  \begin{align*}
    I_1 &= [0,t] \times Y, \\
    T   &= [0, (t), at^2-t] \times \{0\}, \\
    S   &= [at^2, (t+1), (a+1)t^2-1] \times Y, \\
    I_2 &= [2at^2, 2at^2+t] \times Y, \\
    I_3 &= [(3a+1)t^2, (3a+1)t^2+t] \times Y,
  \end{align*}
  and $Y=[0,s_y]$ and $a=4s_y+3$.
\end{definition}

Note that in $I_1,I_2,I_3$ the set of $x$~coordinates is an interval;
in $T$ it is a $t$-step arithmetic progression; and in $S$ it is a
``sparse'' $(t+1)$-step arithmetic progression.

\begin{theorem}
  \label{thm:stacked}
  If $A$ is a stacked Mrose basis, then $|A|=(8s_y+7)t + (3s_y+1)$ and
  $A+A \supseteq (16s_y+14)t^2-1] \times [0,s_y]$.
\end{theorem}

\begin{proof}
  Let us first determine the size of the basis.  We observe that
  $\lvert I_1 \rvert = \lvert I_2 \rvert = \lvert I_3 \rvert =
  (t+1)(s_y+1)$, $\lvert T \rvert = at$, and $\lvert S \rvert =
  t(s_y+1)$.  Because the parts are otherwise disjoint except that
  $I_1 \cap T = \{(0,0),(t,0)\}$, the claim on $|A|$ follows.

  Let us next verify that $A+A$ covers the desired target rectangle.
  We check seven consecutive subrectangles in turn.
  \begin{enumerate}
  \item $[0, at^2-1]              \times Y$ is covered by $I_1 + T$.
  \item $[at^2, (a+1)t^2-1]       \times Y$ is covered by $I_1 + S$.
  \item $[(a+1)t^2, 2at^2-1]      \times Y$ is covered by $T   + S$.
  \item $[2at^2, 3at^2-1]         \times Y$ is covered by $I_2 + T$.
  \item $[3at^2, (3a+1)t^2-1]     \times Y$ is covered by $I_2 + S$.
  \item $[(3a+1)t^2, (4a+1)t^2-1] \times Y$ is covered by $I_3 + T$.
  \item $[(4a+1)t^2, (4a+2)t^2-1] \times Y$ is covered by $I_3 + S$.
  \end{enumerate}
  Because $I_1,I_2,I_3,T,S \subseteq A$, combining observations (1)--(7)
  and $4a+2 = 16s_y+14$
  we have
  \[
  A+A \supseteq [0, (16s_y+14)t^2-1] \times Y
  \]
  as claimed.
\end{proof}

\begin{corollary}
The stacked Mrose basis has efficiency
\[
c = \frac{N}{k^2} = \frac{(16s_y+14)t^2 \cdot (s_y+1)}{\left( (8s_y+7)t \right)^2 + O(t)}
\xrightarrow[t \to \infty]{} \frac{2s_y+2}{8s_y+7}.
\]
\end{corollary}

\begin{example}
  With $s_y=1$, Definition~\ref{def:stacked} gives a basis of size
  $k=15t+4$ for the rectangle $[0, 30t^2-1]\times[0,1]$, with
  efficiency tending to $4/15 > 0.2666$ as $t \to \infty$.
\end{example}

\begin{example}
  With $s_y=2$, Definition~\ref{def:stacked} gives a basis of size
  $k=23t+7$ for the rectangle $[0, 46t^2-1]\times[0,2]$, with
  efficiency tending to $6/23 > 0.2609$ as $t \to \infty$.
  Figure~\ref{fig:thm_1} illustrates this basis in the case of $t=10$.
\end{example}

Although a stacked Mrose basis can be constructed arbitrarily high,
its efficiency tends down to $1/4$ as $s_y$ goes to infinity.  We do
not know whether $1/4$ can be asymptotically exceeded for rectangles
with both dimensions going to infinity (e.g.\ with a constant aspect
ratio).

\begin{figure}[tb]
        \centerline{\includegraphics[width=1\textwidth]{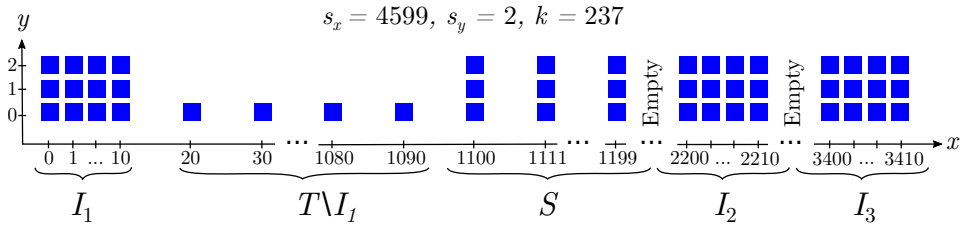}}
	\caption{A schematic illustration of the stacked Mrose basis
          (Definition~\ref{def:stacked}) with parameters $s_y=2$ and
          $t=10$.  In this case $a=11$ and $s_x=4599$.}
        \label{fig:thm_1}
\end{figure}

%%%%%%%%%%%%%%%%%%%%%%%%%%%%%%%%%%%%%%%%%%%%%%%%%%%%%%%%%%%%%%%%%%%%%%
\section{Final remarks}
In this paper, we have studied two dimensional additive bases of minimal cardinality. 
By computation we have listed all minimal bases for
rectangles up to $s_x,s_y\le 11 $ and all minimal restricted bases for
rectangles up to $ s_x,s_y\le 26 $.  Furthermore, we have determined
that the boundary basis is minimal in the restricted case for all
even-sided squares with $ 2 \le s \le 46 $.  We have also found many
non-square solutions for larger $ s_x $.  The L-shaped and boundary
bases are in general not minimal for rectangles; we have presented
three parametric bases that are in general smaller than the trivial
L-shaped and boundary bases.

We note that additive bases are conceptually closely related to
\emph{difference bases}, where the object of interest is the
difference set $A-A$.  One-dimensional difference bases have been
studied e.g.\ by Leech~\cite{leech1956} and
Wichmann~\cite{wichmann1963anote}. Difference bases find applications
in sensor arrays, particularly when second-order statistics of the
element outputs are processed \cite{hoctor1990}. Due to the use of
data covariance in many applications, such as direction-of-arrival
estimation, both one- and two-dimensional difference bases have
received attention recently
\cite{linebarger1993difference,pal2010nested,liu2017}. We also point
out that non-rectangular, for example hexagonal grids have received
some attention in array processing using difference
bases~\cite{haubrich1968array}, and are therefore an interesting
direction of future research for planar additive bases.

%%%%%%%%%%%%%%%%%%%%%%%%%%%%%%%%%%%%%%%%%%%%%%%%%%%%%%%%%%%%%%%%%%%%%%
% Tables at the end

\begin{table}[p]
	\caption{Minimal bases for squares.}
	\label{tab:results_nonrestr_square}
	\centering
		\begin{tabular}{|r|r|r|r|}
			\hline
			$ s $ & $ k $& $ m $ & $ m_\text{u} $\\
			\hline
			0 & 1 & 1 & 1\\
			1 & 3 & 1 & 1\\
			2 & 4 & 1 & 1\\
			3 & 7 & 15 & 10\\
			4 & 8 & 8 & 5\\
			5 & 11 & 137 & 76\\
			6 & 12 & 24 & 14\\
			7 & 14 & 14 & 9\\
			8 & 16 & 103 & 54\\
			9 & 19	&	3531 &	1792\\
			10	 & 20 & 360 & 182\\
			11	& 23 & 26857 & 13465\\
			\hline
		\end{tabular}
\end{table}

\begin{table}[p]
	\caption{Minimal restricted bases for squares.}
	\label{tab:results_restr_square}
	\centering
		\begin{tabular}{|r|r|r|r|}
			\hline
			$ s $ & $ k^* $& $ m $ & $ m_\text{u} $\\
			\hline
			0 & 1 & 1 & 1\\
			2 & 4 & 1 & 1\\
			4 & 8 & 1 & 1\\
			6 &12 &1 & 1\\
			8 & 16 & 9 & 5\\
			10	 & 20 &	17 & 4\\
			12 & 24 & 58 & 16\\
			14	 & 28 & 163 & 28\\
			16 & 32 & 451 & 72\\
			18 & 36	&	2047 &	276\\
			20	 & 40 & 8451 & 1133\\
			22	& 44 & 43807 & 5575\\
			24 & 48	& 213859 & 27108\\
			26	& 52 & 1273607	& 159744\\
			28 & 56	&  & \\	
			30 & 60	&  & \\	
			32 & 64	&  & \\	
			34 & 68	&  & \\	
			36 & 72	&  & \\	
			38 & 76	&  & \\	
			40 & 80	&  & \\	
			42 & 84 &  & \\	
			44 & 88	&  & \\	
			46 & 92	&  & \\
			\hline
		\end{tabular}
\end{table}

\begin{table}[p]
        \caption{Minimal bases for rectangles.}
	\label{tab:results_nonrestr_rect}
	\centering
	\begin{tabular}{|r|r|r|r|r|}
\hline
\textbf{$s_x$}&\textbf{$s_y$}&\textbf{$k$}&\textbf{$\Delta k$}&\textbf{$m_\text{u}$}\\\hline
$0$&$0$&$1$&$0$&$1$\\
\hline
$1$&$0$&$2$&$0$&$1$\\
$$&$1$&$3$&$0$&$1$\\
\hline
$2$&$0$&$2$&$0$&$1$\\
$$&$1$&$4$&$0$&$3$\\
$$&$2$&$4$&$0$&$1$\\
\hline
$3$&$0$&$3$&$0$&$2$\\
$$&$1$&$5$&$0$&$6$\\
$$&$2$&$6$&$0$&$16$\\
$$&$3$&$7$&$0$&$10$\\
\hline
$4$&$0$&$3$&$0$&$2$\\
$$&$1$&$5$&$-1$&$3$\\
$$&$2$&$6$&$0$&$6$\\
$$&$3$&$8$&$0$&$75$\\
$$&$4$&$8$&$0$&$5$\\
\hline
$5$&$0$&$4$&$0$&$5$\\
$$&$1$&$6$&$-1$&$10$\\
$$&$2$&$7$&$-1$&$1$\\
$$&$3$&$9$&$0$&$86$\\
$$&$4$&$10$&$0$&$283$\\
$$&$5$&$11$&$0$&$76$\\
\hline
$6$&$0$&$4$&$0$&$5$\\
$$&$1$&$6$&$-2$&$4$\\
$$&$2$&$8$&$0$&$101$\\
$$&$3$&$9$&$-1$&$1$\\
$$&$4$&$10$&$0$&$16$\\
\hline
\end{tabular}
\begin{tabular}{|r|r|r|r|r|}
\hline
\textbf{$s_x$}&\textbf{$s_y$}&\textbf{$k$}&\textbf{$\Delta k$}&\textbf{$m_\text{u}$}\\\hline
$6$&$5$&$12$&$0$&$660$\\
$$&$6$&$12$&$0$&$14$\\
\hline
$7$&$0$&$4$&$-1$&$2$\\
$$&$1$&$7$&$-2$&$28$\\
$$&$2$&$8$&$-2$&$5$\\
$$&$3$&$10$&$-1$&$25$\\
$$&$4$&$11$&$-1$&$50$\\
$$&$5$&$13$&$0$&$924$\\
$$&$6$&$14$&$0$&$3576$\\
$$&$7$&$14$&$-1$&$9$\\
\hline
$8$&$0$&$4$&$-1$&$1$\\
$$&$1$&$7$&$-3$&$6$\\
$$&$2$&$8$&$-2$&$1$\\
$$&$3$&$11$&$-1$&$325$\\
$$&$4$&$11$&$-1$&$4$\\
$$&$5$&$13$&$-1$&$3$\\
$$&$6$&$14$&$0$&$73$\\
$$&$7$&$15$&$-1$&$16$\\
$$&$8$&$16$&$0$&$54$\\
\hline
$9$&$0$&$5$&$-1$&$11$\\
$$&$1$&$8$&$-3$&$70$\\
$$&$2$&$10$&$-2$&$647$\\
$$&$3$&$12$&$-1$&$1940$\\
$$&$4$&$13$&$-1$&$920$\\
$$&$5$&$15$&$0$&$11479$\\
$$&$6$&$15$&$-1$&$2$\\
\hline
\end{tabular}
\begin{tabular}{|r|r|r|r|r|}
\hline
\textbf{$s_x$}&\textbf{$s_y$}&\textbf{$k$}&\textbf{$\Delta k$}&\textbf{$m_\text{u}$}\\\hline
$9$&$7$&$17$&$0$&$5433$\\
$$&$8$&$18$&$0$&$9171$\\
$$&$9$&$19$&$0$&$1792$\\
\hline
$10$&$0$&$5$&$-1$&$8$\\
$$&$1$&$8$&$-4$&$19$\\
$$&$2$&$10$&$-2$&$174$\\
$$&$3$&$12$&$-2$&$203$\\
$$&$4$&$13$&$-1$&$64$\\
$$&$5$&$15$&$-1$&$267$\\
$$&$6$&$16$&$0$&$357$\\
$$&$7$&$17$&$-1$&$81$\\
$$&$8$&$18$&$0$&$212$\\
$$&$9$&$20$&$0$&$17076$\\
$$&$10$&$20$&$0$&$182$\\
\hline
$11$&$0$&$5$&$-2$&$1$\\
$$&$1$&$9$&$-4$&$258$\\
$$&$2$&$10$&$-4$&$3$\\
$$&$3$&$13$&$-2$&$1368$\\
$$&$4$&$14$&$-2$&$109$\\
$$&$5$&$16$&$-1$&$534$\\
$$&$6$&$17$&$-1$&$96$\\
$$&$7$&$18$&$-1$&$92$\\
$$&$8$&$19$&$-1$&$12$\\
$$&$9$&$21$&$0$&$13860$\\
$$&$10$&$22$&$0$&$42862$\\
$$&$11$&$23$&$0$&$13465$\\
\hline
\end{tabular}

\end{table}

\begin{table}[p]
	\caption{Minimal restricted bases for rectangles.}
	\label{tab:results_restr_rect}
	\centering
	\begin{tabular}{|r|r|r|r|r|}
\hline
\textbf{$s_x$}&\textbf{$s_y$}&\textbf{$k^{*}$}&\textbf{$\Delta k$}&\textbf{$m_\text{u}$}\\\hline
$0$&$0$&$1$&$0$&$1$\\
\hline
$2$&$0$&$2$&$0$&$1$\\
$$&$2$&$4$&$0$&$1$\\
\hline
$4$&$0$&$3$&$0$&$1$\\
$$&$2$&$6$&$0$&$1$\\
$$&$4$&$8$&$0$&$1$\\
\hline
$6$&$0$&$4$&$0$&$1$\\
$$&$2$&$8$&$0$&$1$\\
$$&$4$&$10$&$0$&$1$\\
$$&$6$&$12$&$0$&$1$\\
\hline
$8$&$0$&$4$&$-1$&$1$\\
$$&$2$&$8$&$-2$&$1$\\
$$&$4$&$11$&$-1$&$1$\\
$$&$6$&$14$&$0$&$3$\\
$$&$8$&$16$&$0$&$5$\\
\hline
$10$&$0$&$5$&$-1$&$1$\\
$$&$2$&$10$&$-2$&$2$\\
$$&$4$&$13$&$-1$&$1$\\
$$&$6$&$16$&$0$&$4$\\
$$&$8$&$18$&$0$&$6$\\
$$&$10$&$20$&$0$&$4$\\
\hline
$12$&$0$&$5$&$-2$&$1$\\
$$&$2$&$10$&$-4$&$1$\\
$$&$4$&$14$&$-2$&$2$\\
$$&$6$&$18$&$0$&$14$\\
$$&$8$&$19$&$-1$&$1$\\
$$&$10$&$22$&$0$&$14$\\
$$&$12$&$24$&$0$&$16$\\
\hline
$14$&$0$&$6$&$-2$&$3$\\
$$&$2$&$12$&$-4$&$7$\\
$$&$4$&$16$&$-2$&$15$\\
$$&$6$&$20$&$0$&$91$\\
$$&$8$&$22$&$0$&$47$\\
$$&$10$&$24$&$0$&$30$\\
$$&$12$&$26$&$0$&$37$\\
\hline
\end{tabular}
\begin{tabular}{|r|r|r|r|r|}
\hline
\textbf{$s_x$}&\textbf{$s_y$}&\textbf{$k^{*}$}&\textbf{$\Delta k$}&\textbf{$m_\text{u}$}\\\hline
$14$&$14$&$28$&$0$&$28$\\
\hline
$16$&$0$&$6$&$-3$&$1$\\
$$&$2$&$12$&$-6$&$1$\\
$$&$4$&$16$&$-4$&$1$\\
$$&$6$&$20$&$-2$&$1$\\
$$&$8$&$22$&$-2$&$1$\\
$$&$10$&$26$&$0$&$74$\\
$$&$12$&$28$&$0$&$86$\\
$$&$14$&$30$&$0$&$156$\\
$$&$16$&$32$&$0$&$72$\\
\hline
$18$&$0$&$7$&$-3$&$4$\\
$$&$2$&$14$&$-6$&$20$\\
$$&$4$&$18$&$-4$&$12$\\
$$&$6$&$22$&$-2$&$17$\\
$$&$8$&$25$&$-1$&$34$\\
$$&$10$&$28$&$0$&$279$\\
$$&$12$&$30$&$0$&$286$\\
$$&$14$&$32$&$0$&$302$\\
$$&$16$&$34$&$0$&$345$\\
$$&$18$&$36$&$0$&$276$\\
\hline
$20$&$0$&$7$&$-4$&$2$\\
$$&$2$&$14$&$-8$&$3$\\
$$&$4$&$18$&$-6$&$1$\\
$$&$6$&$22$&$-4$&$1$\\
$$&$8$&$25$&$-3$&$1$\\
$$&$10$&$29$&$-1$&$1$\\
$$&$12$&$32$&$0$&$1155$\\
$$&$14$&$34$&$0$&$1157$\\
$$&$16$&$36$&$0$&$1202$\\
$$&$18$&$38$&$0$&$1406$\\
$$&$20$&$40$&$0$&$1133$\\
\hline
$22$&$0$&$8$&$-4$&$12$\\
$$&$2$&$16$&$-8$&$113$\\
$$&$4$&$20$&$-6$&$14$\\
$$&$6$&$24$&$-4$&$17$\\
\hline
\end{tabular}
\begin{tabular}{|r|r|r|r|r|}
\hline
\textbf{$s_x$}&\textbf{$s_y$}&\textbf{$k^{*}$}&\textbf{$\Delta k$}&\textbf{$m_\text{u}$}\\\hline
$22$&$8$&$28$&$-2$&$381$\\
$$&$10$&$32$&$0$&$8957$\\
$$&$12$&$34$&$0$&$5585$\\
$$&$14$&$36$&$0$&$5601$\\
$$&$16$&$38$&$0$&$5644$\\
$$&$18$&$40$&$0$&$5850$\\
$$&$20$&$42$&$0$&$6705$\\
$$&$22$&$44$&$0$&$5575$\\
\hline
$24$&$0$&$8$&$-5$&$4$\\
$$&$2$&$16$&$-10$&$10$\\
$$&$4$&$20$&$-8$&$1$\\
$$&$6$&$24$&$-6$&$1$\\
$$&$8$&$28$&$-4$&$16$\\
$$&$10$&$32$&$-2$&$50$\\
$$&$12$&$35$&$-1$&$4$\\
$$&$14$&$38$&$0$&$27132$\\
$$&$16$&$40$&$0$&$27177$\\
$$&$18$&$42$&$0$&$27381$\\
$$&$20$&$44$&$0$&$28238$\\
$$&$22$&$46$&$0$&$32680$\\
$$&$24$&$48$&$0$&$27108$\\
\hline
$26$&$0$&$8$&$-6$&$2$\\
$$&$2$&$16$&$-12$&$2$\\
$$&$4$&$22$&$-8$&$46$\\
$$&$6$&$26$&$-6$&$18$\\
$$&$8$&$30$&$-4$&$302$\\
$$&$10$&$34$&$-2$&$1384$\\
$$&$12$&$36$&$-2$&$4$\\
$$&$14$&$40$&$0$&$159771$\\
$$&$16$&$42$&$0$&$159828$\\
$$&$18$&$44$&$0$&$160019$\\
$$&$20$&$46$&$0$&$160874$\\
$$&$22$&$48$&$0$&$165318$\\
$$&$24$&$50$&$0$&$186849$\\
$$&$26$&$52$&$0$&$159744$\\
\hline
\end{tabular}

\end{table}

\begin{table}[p]
	\caption{Minimal restricted bases for $ s_y=2 $.}
	\label{tab:results_restr_sy2}
	\centering
	\begin{tabular}{|r|r|r|}
\hline
\textbf{$s_x$}&\textbf{$k^{*}$}&\textbf{$m_\text{u}$}\\\hline
$2$&$4$&$1$\\
$4$&$6$&$1$\\
$6$&$8$&$1$\\
$8$&$8$&$1$\\
$10$&$10$&$2$\\
$12$&$10$&$1$\\
$14$&$12$&$7$\\
$16$&$12$&$1$\\
$18$&$14$&$20$\\
$20$&$14$&$3$\\
$22$&$16$&$113$\\
$24$&$16$&$10$\\
$26$&$16$&$2$\\
$28$&$18$&$162$\\
$30$&$18$&$22$\\
\hline
\end{tabular}
\begin{tabular}{|r|r|r|}
\hline
\textbf{$s_x$}&\textbf{$k^{*}$}&\textbf{$m_\text{u}$}\\\hline
$32$&$18$&$1$\\
$34$&$20$&$777$\\
$36$&$20$&$50$\\
$38$&$20$&$8$\\
$40$&$20$&$1$\\
$42$&$22$&$412$\\
$44$&$22$&$20$\\
$46$&$24$&$32931$\\
$48$&$24$&$3126$\\
$50$&$24$&$369$\\
$52$&$24$&$37$\\
$54$&$24$&$2$\\
$56$&$26$&$4337$\\
$58$&$26$&$239$\\
$60$&$26$&$36$\\
\hline
\end{tabular}
\begin{tabular}{|r|r|r|}
\hline
\textbf{$s_x$}&\textbf{$k^{*}$}&\textbf{$m_\text{u}$}\\\hline
$62$&$28$&$125247$\\
$64$&$26$&$1$\\
$66$&$28$&$654$\\
$68$&$28$&$62$\\
$70$&$28$&$3$\\
$72$&$28$&$1$\\
$74$&$30$&$2415$\\
$76$&$30$&$97$\\
$78$&$30$&$6$\\
$80$&$30$&$1$\\
$82$&$32$&$18937$\\
$84$&$32$&$1561$\\
$86$&$32$&$193$\\
$88$&$32$&$8$\\
$90$&$32$&$2$\\
\hline
\end{tabular}
\begin{tabular}{|r|r|r|}
\hline
\textbf{$s_x$}&\textbf{$k^{*}$}&\textbf{$m_\text{u}$}\\\hline
$92$&$32$&$1$\\
$94$&$34$&$1284$\\
$96$&$34$&$222$\\
$98$&$34$&$88$\\
$100$&$34$&$1$\\
$102$&$36$&$74170$\\
$104$&$34$&$1$\\
$106$&$36$&$945$\\
$108$&$36$&$242$\\
$110$&$36$&$104$\\
$112$&$38$&$283716$\\
$114$&$38$&$42971$\\
$116$&$36$&$1$\\
$118$&$38$&$454$\\
$120$&$38$&$202$\\
\hline
\end{tabular}

\end{table}

%%%%%%%%%%%%%%%%%%%%%%%%%%%%%%%%%%%%%%%%%%%%%%%%%%%%%%%%%%%%%%%%%%%%%%
\bibliographystyle{plain}
\bibliography{refs}

\end{document}